\documentclass{article}
\usepackage{amssymb,amsmath,graphicx,amsthm,mathrsfs}
\usepackage{color}
\usepackage{mathrsfs}
\usepackage{pdfsync}
\theoremstyle{plain}
\newtheorem{theorem}{Theorem}

\newtheorem{lemma}{Lemma}

\theoremstyle{definition}

\theoremstyle{remark}

\newtheorem{remark}{Remark}
\numberwithin{equation}{section}

\newcommand{\lc}
{\mathrel{\raise2pt\hbox{${\mathop<\limits_{\raise1pt\hbox
{\mbox{$\sim$}}}}$}}}

\newcommand{\gc}
{\mathrel{\raise2pt\hbox{${\mathop>\limits_{\raise1pt\hbox{\mbox{$\sim$}}}}$}}}

\newcommand{\ec}
{\mathrel{\raise2pt\hbox{${\mathop=\limits_{\raise1pt\hbox{\mbox{$\sim$}}}}$}}}

\def \p  {\partial}
\def \d {\Delta}
\def \l {\lambda}

\def \om{\Omega}

\begin{document}
\title{\textbf{Quantitative unique continuation for the heat equation with Coulomb potentials}}
\author{
Can Zhang\thanks{ 
 School of Mathematics and Statistics, Wuhan University, 430072 Wuhan, China. 
Email address: {\it zhangcansx@163.com}. 
}
}

\maketitle


\begin{abstract}
In this paper, we establish a  H\"older-type quantitative estimate  of unique continuation  for solutions to the heat equation with Coulomb potentials in either a bounded convex domain or a $C^2$-smooth bounded domain.  The approach is based on the frequency function method, as well as some parabolic-type Hardy inequalities.

\bigskip



\end{abstract}


\section{Introduction and main results}
This paper is concerned with the quantitative unique continuation  property
for solutions to the heat equation with singular Coulomb potentials at the origin
\begin{equation}\label{6171}
\left\{
\begin{split}
&\p_tu-\d u -\frac{k}{|x|}u=0&\text{in}\;\;\Omega\times(0,+\infty),\\
&u=0&\text{on}\;\;\p\Omega\times(0,+\infty),\\
&u(\cdot,0)\in L^2(\Omega),
\end{split}\right.
\end{equation}
where $k\in\mathbb R$ and $\Omega\subset\mathbb{R}^n$  ($n\geq3$)  is a
bounded Lipschitz  domain which contains the origin.
It is well-known that, for each initial value $u(\cdot,0)\in L^2(\Omega)$ and each $T>0$, Equation \eqref{6171} has a unique solution $u\in C([0,T];L^2(\Omega))\cap L^2(0,T;H_0^1(\Omega))$ (cf., e.g., \cite{Lions2} or \cite{VZ}).

Recall that the space-like strong unique continuation property for any solution $u$ to parabolic equations is as follows.  For any point $x_0\in \Omega$ and any time $t_0>0$, if $u(\cdot,t_0)$ vanishes of infinite order at the point $x_0$ (i.e., 
\begin{equation*}\label{dudu1}
\int_{B_r(x_0)}|u(x,t_0)|^2\,dx=\mathrm o(r^m)\;\;\;\;\text{as} \;\;r\rightarrow0^+,
\end{equation*}
for any positive integer $m$), then $u(\cdot,t_0)\equiv0$ in $\Omega$. Furthermore, if $u$ is zero on the lateral boundary $\partial\Omega\times(0,t_0)$, then $u\equiv 0$ in $\Omega\times(0,t_0)$ by the backward uniqueness property, see   \cite{EF} or \cite{L} for instance.
In other words, when a solution of parabolic equations enjoys such a space-like strong unique continuation property, then it either vanishes in $\Omega$ or  cannot vanishes of infinite order at any point in $\Omega$.
This kind of  unique continuation for solutions to general second order parabolic equations have been established in the works \cite{CXY,EF,EFV,F,L} and  references therein.

Moreover, quantitative estimates of strong unique continuation for second order parabolic equations, such as the doubling property and the two-ball one-cylinder inequality,
have been well understood (see, e.g., \cite{EFV,Phung-Wang1,pw2,PWZ}). We refer to \cite{V1} for a more extensive review on this subject. We also mention that the unique continuation property for stochastic parabolic equations  has been recently studied in \cite{qi1, qi2, xu}.

The aim of this paper is to establish the following quantitative unique continuation:  Given a nonempty open subset $\omega$ of $\Omega$, 
there are
constants $N=N(\Omega,\omega,k,n,T)\geq1$ and
$\alpha=\alpha(\Omega,\omega,k,n)$ with $\alpha\in(0,1)$ such that for any
solution $u$ to Equation~\eqref{6171} and for any $T\in(0,1]$,
\begin{equation}\label{6172}
\|u(\cdot,T)\|_{L^2(\Omega)}
\leq \|u(\cdot,T)\|_{L^2(\omega)}^{\alpha}
\big(N\|u(\cdot,0)\|_{L^2(\Omega)}\big)^{1-\alpha}\,\;\;\;\;\text{for all}\, \,u(\cdot,0)\in L^2(\Omega).
\end{equation}


This kind of H\"{o}lder-type quantitative  estimate of unique continuation was first established in \cite{Phung-Wang1} for the heat equation with bounded potentials in a bounded convex domain. Later on,
it has been extended in \cite{PWZ} to  the case of   bounded domain with a $C^2$-smooth boundary (see also
\cite{BP,pw2}). 
Using sharp analyticity estimates for solutions to general parabolic equations or systems with analytic coefficients,  such a kind of quantitative estimate have been established in a series of recent works \cite{AEWZ,EMS,EMS2}.

The purpose of this paper is twofold. Firstly, although the quantitative estimate \eqref{6172} of unique continuation   for the heat equation with $L^q(\Omega)$ potentials  for any $q>n$ has been established in \cite{pw2},
however, it is still unknown so far for the heat equation  with the Coulomb potential, since the Coulomb potential does not fall into  the class of $L^q(\Omega)$ with some $q>n$. 

Secondly, several applications for the above interpolation estimate \eqref{6172} in Control Theory, such as impulse control, observability inequalities from measurable subsets, and bang-bang properties of optimal controls for parabolic equations, have been recently discussed in \cite{AEWZ,EMS,EMS2,pw2,pwx,PWZ,wz,WZ,WY,Yu,ZB}.

The main results of this paper are included in Theorems \ref{g-t} and \ref{global} below. 
In order to present the basic ideas in our strategy, 
the first one below is for a particular case that the bounded  regular domain $\Omega$ has the convex structure  and the interior observation region $\omega$ is a ball.  Moreover,  one can specify the explicit expression of the dependency of two    constants appearing in  \eqref{6172} for this particular case.

\begin{theorem}\label{g-t}
Let $\Omega$ be a bounded convex domain. Assume $x_0\in\Omega$ and $r>0$ to be such that $B_r(x_0)\subset\Omega$.  Then, there exists a constant $N=N(k,n)\geq1$  such that for any
solution $u$ to Equation~\eqref{6171} and any $0<T\leq 1$,
\begin{equation}\label{1127-6-c}
\begin{split}
\int_{\Omega} |u(x,T)|^2\,dx
&\leq  2\left(\int_{B_{r}(x_0)}\!\!\!|u(x,T)|^2\,dx\right)
^{\alpha(r)}\left(Ne^{\frac{R_\Omega^2}{2T}}\int_\Omega |u(x,0)|^2\,dx\right)^{1-\alpha(r)}\\
\end{split}
\end{equation}
with
\begin{equation*}
\alpha(r)=\frac{1}{1+\frac{32R_\Omega^2}{r^2}
e^{\frac{k^2}{\mu^*}}},
\end{equation*}
where $\mu^*:=(n-2)^2/4$ and $R_\Omega$ is the diameter of $\Omega$.\footnote{Note that $\mu^*$ is the best constant of the  Hardy inequality.} 
\end{theorem}

Next, we turn to state the main result for  the general class of $C^2$-smooth domains.

\begin{theorem}\label{global}
Let $\Omega$ be a bounded domain with a $C^2$-smooth boundary, and let $\omega\subset\Omega$ be a non-empty open subset. Then, there are constants $N=N(\Omega,\omega,k,n)\geq 1$ and
$\alpha=\alpha(\Omega,\omega,k,n)$ with $\alpha\in(0,1)$ such that  for any
solution $u$ to Equation~\eqref{6171} and any $0<T\leq 1$,
\begin{equation}\label{cccc1}
\int_{\Omega}|u(x,T)|^2\;dx\leq
\left(\int_{\omega}|u(x,T)|^2\;dx\right)^{\alpha}
\left(Ne^{\frac{N}{T}}\int_{\Omega}|u(x,0)|^2\;dx\right)
^{1-\alpha}.
\end{equation}
\end{theorem}

\medskip

The proofs of Theorems \ref{g-t} and \ref{global} are based on the weighted frequency function method and two parabolic-type  Hardy inequalities. The frequency function method is well known in the studies of   quantitative estimates of unique continuation for the second elliptic equations, such as the three-ball inequality and the doubling property (cf., e.g., \cite{Lin1}, \cite{Lin2}, \cite{K} and references therein). The extension of this method to the parabolic equations is first made in \cite{Poon}, where the strong unique continuation property  of parabolic equations in the whole space $\mathbb R^n$ was built up. Later on, Escauriaza, et al.,  developed  this method to deduce quantitative estimates of unique continuation   for general second order parabolic equations (see, e.g., \cite{E1}, \cite{EF}, \cite{EFV}, \cite{F} and \cite{Phung-Wang1}). 

Compared with the context  in the earlier  works \cite{Phung-Wang1, pw2, PWZ},  the new difficulty here is how to deal with the singular 
lower-order term.  The novelty  of this paper is to apply some parabolic-type  Hardy inequalities (see Section \ref{ph} below) 
to overcome this difficulty.

The rest of this paper  is organized as follows.  In Section \ref{sec2}, we  apply the global frequency function method to deduce the interpolation inequality \eqref{6172} when $\Omega$ is a bounded and convex domain (i.e., Theorem \ref{g-t}). In Section \ref{sec3}, we will show how to extend it to a $C^2$-smooth bounded domain by localized frequency function method (i.e., Theorem \ref{global}). Finally, we conclude the paper with several comments in Section \ref{sec4}.

\bigskip

\section{Global frequency function method: Proof of Theorem~\ref{g-t}}\label{sec2}

\subsection{Monotonicity property of frequency function}
 For each $\lambda>0$, let us set
\begin{equation*}
G_\l(x,t)=(T-t+\lambda)^{-n/2}e^{-\frac{|x|^2}{4(T-t+\lambda)}},\;
(x,t)\in \mathbb{R}^n\times[0,T],
\end{equation*}
which is the backward caloric function in $
\mathbb{R}^n\times[0,T]$:
$$
(\partial_t +\Delta) G_\l(x,t)=0.
$$
Given any $x_0\in\Omega$, let us write
\begin{equation*}\label{6201}
G_{\lambda, x_0}(x,t)=G_\l(x-x_0,t),\; (x,t)\in
\mathbb{R}^n\times[0,T],\; \l>0.
\end{equation*}
 Let $f\in L^2(\Omega\times(0,T))$ and let $u$ be a
solution of
\begin{equation}\label{coupled1}
\left\{
\begin{split}
&\p_tu-\d u=f(x,t)\;&\text{in}\;\;\Omega\times(0,T),\\
&u=0\;\;\;\;\;\;&\text{on}\;\;\p\Omega\times(0,T),\\
&u(\cdot,0)\in L^2(\Omega).
\end{split}\right.
\end{equation}
Associated with each triplet  $(u, f, x_0)$ (where $f\in
L^2(\Omega\times(0,T))$, $u$ solves (\ref{coupled1}) and
$x_0\in\Omega$), we define  the weighted
frequency function
\begin{equation*}\label{frequency}
N_\l(t)=\frac{I_\l(t)}{H_\l(t)}
\end{equation*}
for all $t\in \{t\in (0,T]; H_\l(t)\neq
0\}$,
where
\begin{equation*}
\begin{split}
I_\l(t)&=\int_{\Omega}|\nabla u(x,t)|^2G_{\l,x_0}(x,t)\,dx,\; t\in (0,T],\; \lambda>0,\\
H_\l(t)&=\int_{\Omega}|u(x,t)|^2G_{\l, x_0}(x,t)
\,dx,\; t\in (0,T],\; \lambda>0.\\
\end{split}
\end{equation*}

We begin with the following lemma, which has been proved  in
\cite{Phung-Wang1}. For the sake of the completeness of the paper,
we provide a probably simple proof here.
\begin{lemma}\label{monotonicity}
Let $f\in L^2(\Omega\times(0,T))$, $u$ be a solution of
\eqref{coupled1} and $x_0\in\Omega$. Then the weighted frequency
function associated with $(u, f, x_0)$ has the following properties:

\noindent \textnormal{(i)}. For each $t\in(0,T]$ with $u(\cdot,t)\neq 0$ in
$L^2(\Omega)$,
\begin{equation*}
N_\l(t)=-\frac{1}{2}\frac{d}{dt}\log(H_\l(t))+
\frac{\int_{\Omega}ufG_{\l,x_0}\,dx}{\int_{\Omega}u^2G_{\l,x_0}
\,dx}\,.
\end{equation*}

\noindent \textnormal{(ii)}. When $\Omega$ is either convex or star-shaped
with respect to $x_0$,
\begin{equation*}
\frac{d}{dt}N_\l(t)\leq
\frac{N_\l(t)}{T-t+\lambda}+\frac{\int_{\Omega}f^2G_{\l,x_0}\,dx}
{H_\l(t)}\,
\end{equation*}
 for each $t\in(0,T]$ with $u(\cdot,t)\neq 0$ in
$L^2(\Omega)$.
\end{lemma}

\begin{proof}
 By the Green formula, we have
\begin{equation*}
\begin{split}
\frac{d}{dt}H_\l(t)&=\int_\Omega2u\p_tuG_{\l,x_0}+u^2\p_tG_{\l,x_0}\,dx\\
&=2\int _\Omega u(\p_tu-\Delta u)G_{\l,x_0}\,dx-2\int_\Omega|\nabla u|^2G_{\l,x_0}\,dx\\
&=2\int_\Omega ufG_{\l,x_0}\,dx-2I_\l(t),
\end{split}
\end{equation*}
which leads to  \textnormal{(i)}.\\

We now prove  \textnormal{(ii)}. By the integration by parts, it follows that
\begin{equation}\label{1117-1}
\begin{split}
\frac{d}{dt}H_\l(t)&=2\int_\Omega\Big(\p_tu-\frac{\nabla u\cdot (x-x_0)}{2(T-t+\lambda)}-\frac{f}{2}\Big)uG_{\lambda,x_0}\,dx+
\int_\Omega fuG_{\lambda,x_0}\,dx,\\
\end{split}
\end{equation}

\begin{equation}\label{1117-2}
\begin{split}
I_\l(t)&=-\int_\Omega\Big(\p_tu-\frac{\nabla u\cdot (x-x_0)}{2(T-t+\lambda)}-\frac{f}{2}\Big)uG_{\lambda,x_0}\,dx+
\frac{1}{2}\int_\Omega fuG_{\lambda,x_0}\,dx.\\
\end{split}
\end{equation}
Meanwhile,
\begin{equation}\label{1117-4}
\begin{split}
\frac{d}{dt}I_\l(t)&=2\int_\Omega \nabla u\cdot \nabla (\p_tu)G_{\lambda,x_0}\,dx
+\int_\Omega\nabla(|\nabla u|^2)\cdot \nabla G_{\lambda,x_0}\,dx+\mathcal{A},\\
\end{split}
\end{equation}
where
\begin{equation*}\label{1117-3}
\mathcal{A}=\int_{\p\Omega}\frac{(x-x_0)\cdot \nu_x}{2(T-t+\lambda)}
\left(\frac{\p u}{\p \nu}\right)^2G_{\lambda,x_0}\,d\sigma.
\end{equation*}
 Since
\begin{equation}\label{1117-6}
\begin{split}
\nabla(|\nabla u|^2)\cdot\nabla G_{\lambda,x_0}=\frac{-1}{T-t+\l}
\Big(\nabla u\cdot\nabla\big(\nabla u\cdot (x-x_0)\big)-
|\nabla u|^2\Big)G_{\lambda,x_0},
\end{split}
\end{equation}
it follows from \eqref{1117-4} and \eqref{1117-6} that
\begin{equation}\label{1117-8}
\begin{split}
\frac{d}{dt}I_\l(t)&=2\int_{\Omega}\nabla u\cdot\nabla\Big(\p_t u-\frac{\nabla u\cdot(x-x_0)}{2(T-t+\l)}\Big)G_{\lambda,x_0}\,dx\\
&\;\;\;\;+\frac{I_\l(t)}{T-t+\l}+\mathcal{A}.
\end{split}
\end{equation}
On the other hand, we have by the Green formula,
\begin{equation}\label{1117-10}
\begin{split}
&\int_{\Omega}\nabla u\cdot\nabla\Big(\p_t u-\frac{\nabla u\cdot (x-x_0)}{2(T-t+\l)}\Big)G_{\lambda,x_0}\,dx\\
&=-\mathcal{A}-\int_{\Omega}\Big(\p_tu-\frac{\nabla u\cdot (x-x_0)}{2(T-t+\l)}\Big)\Big(\p_tu-\frac{\nabla u\cdot (x-x_0)}{2(T-t+\l)}-f\Big)G_{\lambda,x_0}\,dx.\\
\end{split}
\end{equation}
From \eqref{1117-8} and \eqref{1117-10}, it holds that
\begin{equation}\label{1117-9}
\begin{split}
\frac{d}{dt}I_\l(t)&=-2\int_{\Omega}\Big(\p_t u-\frac{\nabla u\cdot (x-x_0)}{2(T-t+\l)}-\frac{f}{2}\Big)^2G_{\lambda,x_0}\,dx+\frac{1}{2}
\int_\Omega f^2G_{\lambda,x_0}\,dx\\
&\;\;\;\;+\frac{I_\l(t)}{T-t+\l}-\mathcal{A}.
\end{split}
\end{equation}

Finally, we obtain from \eqref{1117-1}, \eqref{1117-2} and \eqref{1117-9} that
\begin{equation*}
\begin{split}
\frac{d}{dt}N_\l(t)&=\frac{\frac{d}{dt}\big(I_\l(t)\big)H_\l(t)
-I_\l(t)
\frac{d}{dt}H_\l(t)}{H^2_\l(t)}\\
&=\frac{N_\l(t)}{T-t+\l}+\frac{\int_\Omega f^2G_{\lambda,x_0}\,dx}{2H_\l(t)}-\frac{\mathcal{A}}
{H_\l(t)}-\frac{\Big(\int_\Omega fuG_{\lambda,x_0}\,dx\Big)^2}{2H_\l^2(t)}\\
&\;\;\;\;-\frac{2}{H_\l(t)}\int_\Omega \Big(
\p_tu-\frac{\nabla u\cdot (x-x_0)}{2(T-t+\l)}-\frac{f}{2}\Big)^2G_{\lambda,x_0}\,dx\\
&\;\;\;\;+\frac{2}{H^2_\l(t)}\int_\Omega \Big(
\p_tu-\frac{\nabla u\cdot (x-x_0)}{2(T-t+\l)}-\frac{f}{2}\Big)uG_{\lambda,x_0}\,dx.\\
\end{split}
\end{equation*}
Since $\Omega$ is either convex or star-shaped with respect to $x_0$,  it holds
that $(x-x_0)\cdot\nu_x\geq 0$ when $x\in\partial\Omega$ (see for instance  \cite[pp. 515]{evans}). (Here,
$\nu_x$ is the outward normalized vector of $\p\Omega$ at $x$.)
Consequently, $\mathcal{A}\geq 0$. By the Cauchy-Schwartz
inequality,
\begin{multline*}
\left[\int_\Omega\Big(
\p_t u-\frac{\nabla u\cdot (x-x_0)}{2(T-t+\l)}-\frac{f}{2}\Big)u G_{\lambda,x_0}\,dx\right]^2
\\
\leq \int_\Omega \Big(
\p_t u-\frac{\nabla u\cdot (x-x_0)}{2(T-t+\l)}-\frac{f}{2}\Big)^2G_{\lambda,x_0}\,dx
\times
\int_\Omega u^2G_{\lambda,x_0}\,dx,
\end{multline*}
which leads to
\begin{equation*}
\frac{d}{dt}N_\l(t)\leq \frac{N_\l(t)}{T-t+\l}+\frac{\int_\Omega
f^2G_{\lambda,x_0}\,dx}{2H_\l(t)},
\end{equation*}
i.e.,  \textnormal{(ii)} stands. 
\end{proof}

\subsection{Two parabolic-type Hardy
inequalities}\label{ph}

  We next introduce the following two parabolic-type Hardy
inequalities, which are crucial in the proof of main results.
\begin{lemma}\label{hardytype}
Let $x_0\in \Omega$ and $\mu^*:=(n-2)^2/4$, $n\geq3$. Then for each
$\lambda>0$ and each $\varphi\in H^1_0(\Omega)$, it holds that
\begin{equation}\label{pei3}
\begin{split}
\frac{1}{16\l^2}\int_{\Omega}|x-x_0|^2\varphi^2e^{-\frac{|x-x_0|^2}{4\l}}\,dx
\leq \int_{\Omega}\Big(|\nabla \varphi|^2-\frac{\mu^*}{|x|^2}\varphi^2\Big)e^{-\frac{|x-x_0|^2}{4\l}}\,dx\\
+\frac{n}{4\l} \int_{\Omega}\varphi^2e^{-\frac{|x-x_0|^2}{4\l}}\,dx.
\end{split}
\end{equation}
\end{lemma}
\begin{proof}
Recall the well-known Hardy inequality (cf., e.g., \cite{VZ}):
\begin{equation}\label{pei1}
\mu^*\int_{\Omega}\frac{g^2}{|x|^2}\,dx\leq \int_{\Omega}|\nabla
g|^2\,dx\;\;\mbox{for all}\; \; g\in H^1_0(\Omega).
\end{equation}
Given $\varphi\in H^1_0(\Omega)$, $\lambda>0$ and $x_0\in \Omega$,  let
$g=\varphi e^{-\frac{|x-x_0|^2}{8\l}}$. Clearly, $g\in H_0^1(\Omega)$ and
\begin{equation*}
\nabla g=\nabla \varphi e^{-\frac{|x-x_0|^2}{8\l}}-\frac{x-x_0}{4\l}
\varphi e^{-\frac{|x-x_0|^2}{8\l}}.
\end{equation*}
By \eqref{pei1},
\begin{equation*}\label{pei2}
\begin{split}
\mu^*\int_{\Omega}\frac{\varphi^2}{|x|^2}e^{-\frac{|x-x_0|^2}{4\l}}\,dx
&\leq \int_{\Omega}|\nabla \varphi|^2 e^{-\frac{|x-x_0|^2}{4\l}}
+\frac{|x-x_0|^2}{16\l^2}\varphi^2e^{-\frac{|x-x_0|^2}{4\l}}\,dx\\
&\;\;\;\;-\frac{1}{2\l}\int_{\Omega}(x-x_0)\cdot\nabla \varphi \varphi
e^{-\frac{|x-x_0|^2}{4\l}}\,dx.
\end{split}
\end{equation*}
By the integration  by  parts, we have
\begin{equation}\label{pei4}
\int_{\Omega}(x-x_0)\cdot\nabla \varphi \varphi
e^{-\frac{|x-x_0|^2}{4\l}}\,dx
=-\frac{1}{2}\int_{\Omega}\Big(n-\frac{|x-x_0|^2}{2\l}\Big)\varphi^2
e^{-\frac{|x-x_0|^2}{4\l}}\,dx.
\end{equation}
The last two inequalities imply \eqref{pei3}. 
\end{proof}

\begin{lemma}\label{morepei}
Let $x_0\in \Omega$ and $\mu^*$ be as above. Then for each
$m\geq0$ and each  $\gamma\in(0,2)$, there exists a constant
$C=C(m,\gamma)>0$ such that when $\lambda>0$ and $ \varphi\in
H_0^1(\Omega)$,
\begin{equation}\label{wang2.14}
\begin{split}
\frac{1}{16\l^2}\int_{\Omega}|x-x_0|^2
\varphi^2e^{-\frac{|x-x_0|^2}{4\l}}\,dx & \leq
\int_{\Omega}\Big[|\nabla
\varphi|^2-\frac{\mu^*}{|x|^2}\varphi^2-\frac{m}{|x|^\gamma}\varphi^2\Big]
e^{-\frac{|x-x_0|^2}{4\lambda}}\,dx\\
&\;\;\;\;+\Big[\frac{n}{4\lambda}+C\Big]\int_{\Omega}\varphi^2
e^{-\frac{|x-x_0|^2}{4\lambda}}\,dx.
\end{split}
\end{equation}
\end{lemma}
\begin{proof}
Given $x_0\in \Omega$,  $\lambda>0$ and $\varphi\in H^1_0(\Omega)$,
let $z=\varphi e^{-\frac{|x-x_0|^2}{8\lambda}}$. Then,
\begin{equation}\label{1551}
\nabla z=\nabla \varphi
e^{-\frac{|x-x_0|^2}{8\lambda}}-\frac{(x-x_0)}{4\lambda} \varphi
e^{-\frac{|x-x_0|^2}{8\lambda}}
\end{equation}
and
\begin{equation}\label{1552}
|\nabla z|^2=|\nabla \varphi|^2
e^{-\frac{|x-x_0|^2}{4\lambda}}+\frac{|x-x_0|^2}{16\lambda^2}\varphi^2
e^{-\frac{|x-x_0|^2}{4\lambda}} -\frac{(x-x_0)\cdot\nabla
\varphi}{2\lambda}\varphi e^{-\frac{|x-x_0|^2}{4\lambda}}.
\end{equation}
This, together with \eqref{pei4}, \eqref{1551}, \eqref{1552}, and the following improved version of Hardy inequality
(cf., for instance, \cite[(2.15)]{VZ})
\begin{equation*}
m\int_{\Omega}\frac{z^2}{|x|^\gamma}\,dx
\leq \int_{\Omega}\Big[|\nabla z|^2-\frac{\mu^*}{|x|^2}z^2\Big]\,dx
+C(m,\gamma)\int_{\Omega}z^2\,dx,\;\;\;\forall z\in H_0^1(\Omega),
\end{equation*}
leads to (\ref{wang2.14}).
\end{proof}

\begin{remark}
Note that in the proofs of Lemmas \ref{hardytype} and \ref{morepei} we do not 
use the convexity of the domain $\Omega$. In fact, they are still valid whenever $\Omega$  is a bounded 
Lipschitz domain in $\mathbb R^n$, $n\geq3$.

\end{remark}

\subsection{The proof of Theorem~\ref{g-t}}
We first  apply the above-mentioned three lemmas to obtain the
following  weighted estimate for solutions to  Equation \eqref{6171}.

\begin{lemma}\label{peiweightestimate}
Assume the bounded domain $\Omega$ is convex. Let $x_0\in\Omega$ and let $0<\lambda\leq T\leq 1$.
 Then for each $r>0$ with $B_r(x_0)\subset\Omega$ and each solution $u$ to Equation~\eqref{6171}, it holds that
\begin{multline}\label{pei0}
\int_{\Omega}|u(x,T)|^2e^{-\frac{|x-x_0|^2}{4\l}}\,dx
\leq \int_{B_r(x_0)}|u(x,T)|^2e^{-\frac{|x-x_0|^2}{4\l}}\,dx\\
+\frac{64\l}{r^2}e^{\frac{k^2}{\mu^*}}
\left[\log
\frac{\int_\Omega |u(x,0)|^2\,dx}{\int_\Omega |u(x,T)|^2\,dx}+
\mathcal{C}_1\right]
\int_{\Omega}|u(x,T)|^2
e^{-\frac{|x-x_0|^2}{4\l}}\,dx,
\end{multline}
where
\begin{equation*}\label{2911}
\mathcal{C}_1=\mathcal{C}_1(T,k,n,\mu^*,R_\Omega)=\frac{R_\Omega^2}{2T}+C(k)
+\frac{n}{4}\big(1+\frac{k^2}{\mu^*}\big)+n.
\end{equation*}\end{lemma}
\medskip
Here and in the sequel,  
$R_\Omega$ is the diameter of $\Omega$, and  $C(\cdot)$ denotes a positive constant depending only on what are enclosed in the brackets and it may change from line to line in the context.

\begin{proof}
 We only need to show the desired estimate \eqref{pei0} for an arbitrarily fixed solution  $u$  to
Equation \eqref{6171} with $u(\cdot, 0)\neq 0$ in $L^2(\Omega)$. Let
$N_\lambda(\cdot)$ be the weighted frequency function associated with $(u,
f, x_0)$  where $f=ku/|x|$. It follows from Lemma 2.5 in \cite{pw2}
that
\begin{multline}\label{pei5}
\int_\Omega|u(x,T)|^2
e^{-\frac{|x-x_0|^2}
{4\l}}\,dx\leq\int_{B_r(x_0)}|u(x,T)|^2e^{-\frac{|x-x_0|^2}{4\l}}\,dx
\\
+\frac{16\l}{r^2}
\Big[\l N_\l(T)+\frac n4\Big]
\int_\Omega|u(x,T)|^2
e^{-\frac{|x-x_0|^2}
{4\l}}\,dx.
\end{multline}

Next,
we apply Lemmas~\ref{monotonicity}, \ref{hardytype} and \ref{morepei} to deduce a bound for
the quantity $\l N_\l(T)$ when $0<\l\leq T\leq1$. Since $u(\cdot, 0)\neq 0$ in
$L^2(\Omega)$, by backward uniqueness it holds that $u(\cdot, t)\neq 0$ in $L^2(\Omega)$,
for each $t\in [0,T]$.
 From  (ii) of
Lemma~\ref{monotonicity} (where  $(u,f)=(u, ku/|x|)$), we have
\begin{equation}\label{Wang2.20}
\frac{d}{dt}N_\l(t)\leq \frac{N_\l(t)}{T-t+\l}
+\frac{k^2\int_{\Omega}\frac{u^2}{|x|^2}G_{\l,x_0}\,dx}{H_\l(t)}\;\;\;\mbox{for
each}\;\; t\in (0,T].
\end{equation}
By  Lemma~\ref{hardytype} (where $\varphi=u$ and $\l=T-t+\l$ with
$t\in (0,T]$), it follows that
\begin{equation*}
k^2\int_\Omega\frac{u^2}{|x|^2}G_{\l.x_0}dx\leq
\frac{k^2}{\mu^*}\int_\Omega |\nabla
u|^2G_{\l,x_0}dx+\frac{k^2n}{4\mu^*(T-t+\l)}\int_\Omega
u^2G_{\l,x_0}dx,\; t\in (0,T].
\end{equation*}
This, along with (\ref{Wang2.20}), indicates that
\begin{equation*}\label{21-2}
\frac{d}{dt}N_\l(t)\leq
\frac{N_\l(t)}{T-t+\lambda}+\frac{k^2}{\mu^*}
N_\lambda(t)+\frac{nk^2}{4\mu^*(T-t+\lambda)},\; t\in (0,T].
\end{equation*}
Consequently,
\begin{equation*}
\frac{d}{dt}\big[(T-t+\lambda)N_\l(t)\big]\leq
\frac{k^2}{\mu^*}(T-t+\l)N_\l(t)+\frac{nk^2}{4\mu^*}, \; t\in (0,T].
\end{equation*}
This yields
\begin{equation*}\label{295}
\frac{d}{dt}\Big[e^{-tk^2/\mu^*}(T-t+\lambda)N_\lambda(t)\Big]
\leq \frac{nk^2}{4\mu^*}e^{-tk^2/\mu^*},\;t\in(0,T].
\end{equation*}
Integrating the above inequality from $t$ to $T$, we get that
\begin{equation*}
e^{-Tk^2/\mu^*}\lambda N_\lambda(T)\leq e^{-tk^2/\mu^*}(T-t+\lambda)N_\lambda(t)
+\frac{nk^2}{4\mu^*}\int_t^Te^{-sk^2/\mu^*}\,ds,\,t\in(0,T].
\end{equation*}
Thus,
\begin{equation*}\label{297}
e^{-Tk^2/\mu^*}\lambda N_\lambda(T)\leq (T+\lambda)N_\lambda(t)
+\frac{nk^2}{4\mu^*},\,t\in(0,T].
\end{equation*}
Integrating  the above inequality with respect to $t$ over $(0,T/2)$, we obtain that
\begin{equation}\label{pei11}
e^{-Tk^2/\mu^*}\l N_\l(T)\leq \frac{2(T+\l)}{T}
\int_{0}^{T/2}N_\l(t)\,dt+\frac{nk^2}{4\mu^*}.
\end{equation}

On the other hand, by (i) of Lemma~\ref{monotonicity} (where $(u,f)=
(u,ku/|x|)$), we see  that
\begin{equation}\label{pei7}
N_\l(t)=-\frac{1}{2}\frac{d}{dt}\log \big(H_\l(t)\big)+\frac{\int_{\Omega}\frac{k}{|x|}u^2G_{\l,x_0}\,dx}
{H_\l(t)}.
\end{equation}
By Lemma~\ref{morepei} (where $\varphi=u$, $m=2|k|$, $\gamma=1$ and $\lambda=T-t+\lambda$), it stands that
\begin{equation*}\label{21-1}
2|k|\int_{\Omega}\frac{u^2}{|x|}
G_{\l,x_0}\,dx
\leq \int_{\Omega}|\nabla u|^2G_{\l,x_0}\,dx
+\Big[\frac{n}{4(T-t+\lambda)}+C(k)\Big]\int_{\Omega}u^2
G_{\l,x_0}\,dx.
\end{equation*}
This, along with \eqref{pei7}, implies that
\begin{equation*}
N_\l(t)\leq -\frac{d}{dt}\log \big(H_\l(t)\big)+
\Big[C(k)+\frac{n}{4(T-t+\l)}
\Big].
\end{equation*}
Integrating the above inequality over $(0,T/2)$, we get that
\begin{equation}\label{pei9}
\int_{0}^{T/2}N_\l(t)\,dt\leq\log \frac{H_\l(0)}{H_\l(T/2)}+\frac{T}{2}\Big[C(k)+\frac{n}{2T}
\Big].
\end{equation}
Notice that
\begin{equation*}\label{2913}
\begin{split}
\frac{H_\l(0)}{H_\l(T/2)}&=\frac{(T+\l)^{-n/2}\int_\Omega |u(x,0)|^2e^{-\frac{|x-x_0|^2}{4(T+\l)}}\,dx}
{(\frac{T}{2}+\l)^{- n/2}\int_\Omega |u(x,T/2)|^2e^{-\frac{|x-x_0|^2}{4(\frac{T}{2}+\l)}}\,dx}\\
&\leq e^{\frac{R_\Omega^2}{2T}}\frac{\int_\Omega |u(x,0)|^2\,dx}{\int_\Omega |u(x,T/2)|^2\,dx}.
\end{split}
\end{equation*}
From which and the standard  energy estimate
\begin{equation*}\label{2914}
\int_\Omega|u(x,T)|^2\,dx\leq
e^{C(k)T}\int_\Omega |u(x,T/2)|^2\,dx,
\end{equation*}
we have
\begin{equation}\label{pei10}
\frac{H_\l(0)}{H_\l(T/2)}\leq e^{\frac{R_\Omega^2}{2T}+C(k)}\frac{\int_\Omega |u(x,0)|^2\,dx}{\int_\Omega |u(x,T)|^2\,dx}.
\end{equation}
Therefore, it follows from \eqref{pei11}, \eqref{pei9} and  \eqref{pei10} that when $0<\l\leq T\leq1$
\begin{equation*}\label{pei12}
\begin{split}
\l N_\l(T)&\leq
4e^{\frac{k^2}{\mu^*}}
\left[\log \frac{\int_\Omega |u(x,0)|^2\,dx}{\int_\Omega |u(x,T)|^2\,dx}+
\frac{R_\Omega^2}{2T}+C(k)+\frac{n}{4}\big(1+\frac{k^2}
{\mu^*}\big)\right].
\end{split}
\end{equation*}
This, combined with \eqref{pei5}, arrives at the desired estimate
\eqref{pei0}.
\end{proof}

\bigskip

We end this section with  the proof of Theorem~\ref{g-t}.
\begin{proof}[The proof of Theorem~\ref{g-t}]

Let us now choose
\begin{equation*}\label{1127-1-c}
\l_0=\frac{r^2}{128}e^{-\frac{k^2}{\mu^*}}
\left[\log \frac{\int_\Omega |u(x,0)|^2\,dx}{\int_\Omega |u(x,T)|^2\,dx}+\mathcal{C}_1\right]^{-1}
\end{equation*}
with the same constant $\mathcal C_1$ given in Lemma \ref{peiweightestimate}.
It is easy to check that  $0<\l_0<T$.
According to Lemma~\ref{peiweightestimate} (where $\lambda=\lambda_0$), it holds that
\begin{equation*}\label{1123-2-c}
\begin{split}
&\frac{1}{2}\int_\Omega|u(x,T)|^2
e^{-\frac{|x-x_0|^2}
{4\l_0}}\,dx
\leq\int_{B_r(x_0)}|u(x,T)|^2e^{-\frac{|x-x_0|^2}{4\l_0}}\,dx.\\
\end{split}
\end{equation*}
Which implies 
\begin{equation}\label{1117-20-c}
\int_\Omega|u(x,T)|^2\,dx\leq
2e^{\frac{R_\Omega^2}{4\l_0}}\int_{B_r(x_0)}\!\!|u(x,T)|^2\,dx.
\end{equation}
Notice that
\begin{equation*}
\begin{split}
e^{\frac{R_\Omega^2}{4\l_0}}&=e^{\frac{32R_\Omega^2e^{\frac{k^2}{\mu^*}}}{r^2}\left[\log \frac{\int_\Omega |u(x,0)|^2\,dx}{\int_\Omega |u(x,T)|^2\,dx}+\mathcal{C}_1\right]}=\left(e^{\mathcal{C}_1}
\frac{\int_\Omega |u(x,0)|^2\,dx}{\int_\Omega |u(x,T)|^2\,dx}\right)
^{\frac{32R_\Omega^2e^{\frac{k^2}{\mu^*}}}{r^2}}.
\end{split}
\end{equation*}
This, along with \eqref{1117-20-c}, leads to
\begin{equation*}
\begin{split}
&\int_\Omega|u(x,T)|^2
\,dx\leq 2\left(e^{\mathcal{C}_1}\frac{\int_\Omega |u(x,0)|^2\,dx}{\int_\Omega |u(x,T)|^2\,dx}\right)
^{\frac{32R_\Omega^2e^{\frac{k^2}{\mu^*}}
}{r^2}}\int_{B_r(x_0)}\!\!|u(x,T)|^2\,dx.
\end{split}
\end{equation*}
Which implies \eqref{1127-6-c} and completes the proof.
\end{proof}

\begin{remark}
By following the argument in \cite{EFV} and the facts established above,  one can easily obtain  the following doubling property:  For any $u(\cdot,0)\in L^2(\Omega)$, there exists a constant $C$,  independent of $r$ and $x_0$, such that 
\begin{equation*}\label{1123-11-c}
\begin{split}
&\int_{B_{2r}(x_0)} |u(x,T)|^2\,dx\leq  C
\int_{B_r(x_0)}|u(x,T)|^2\,dx,  \\
\end{split}
\end{equation*}
for all $r>0$ such that $B_{2r}(x_0)\subset\Omega$. From which one can derive the space-like strong unique continuation property for solutions to \eqref{6171}.
\end{remark}

\bigskip

\section{Local frequency function method: Proof of Theorem~\ref{global}}\label{sec3}

We sketch  the main idea of the proof of Theorem~\ref{global} as follows. 
First, we apply the frequency function method (as in Section \ref{sec2}) in a star-shaped sub-domain to deduce a localized version of \eqref{6172}, which means that the left hand side of  \eqref{6172} is replaced by the local energy of the solution at the time $T$. Then,
by iterating the above-mentioned localized version and the standard argument of propagation of smallness, as well as a finite covering argument, we can conclude the desired inequality \eqref{6172} when $\Omega$ is a bounded and $C^2$-smooth domain. Noting that any bounded and $C^2$-smooth domain can be covered by finite numbers of star-shaped domains (see, e.g., \cite[Theorem 8]{AEWZ}),

\subsection{Backward estimates}
We start with a version of locally backward energy estimate for solutions to Equation~\eqref{6171}, which is similar to \cite[Lemma 3]{PWZ}.
\begin{lemma}\label{backwardenergyestimate}
Let $x_0\in\Omega$, $0<T\leq 1$, $R\in(0,1]$ and $\delta\in(0,1]$. Then
there exist constants $C(k)$ and $C_1(k)$ such that 
 for any solution $u$ of Equation \eqref{6171} with  $u(\cdot,0)=u_0\in L^2(\Omega)\setminus\{0\}$, 
\begin{equation}\label{ccan4}
\|u_0\|^2_{L^2(\Omega)}
\leq C(k)e^{\frac{1}{h_0}} \|u(\cdot,t)\|^2_{L^2(\Omega\cap B_{(1+\delta)R}(x_0))},\;\;\text{when}\;\;
t\in[T-h_0,T]
\end{equation}
with $h_0$ given by the equality
\begin{equation}\label{ch}
h_0=\frac{\delta^3R^2}{8(1+\delta)^2\log\Big[
\frac{C_1(k)}{\delta^2R^2}e^{\frac{1}{T}}\frac
{\|u_0\|^2_{L^2(\Omega)}}{\|u(\cdot,T)\|^2_{L^2(\Omega\cap B_R(x_0))}}\Big]}\,.
\end{equation}
\end{lemma}

\begin{proof}
We carry out the proof into three steps.\\

Step 1. Choose a suitable multiplier.\\

Let $\psi\in C^\infty_0(\Omega\cap B_{(1+\delta)R}(x_0))$ be the cut-off function verifying
\begin{equation*}
0\leq\psi\leq 1\;\text{in}\; \Omega\cap B_{(1+\delta)R}(x_0), \;\psi\equiv1 \;\text{in}\;\;\Omega\cap B_{(1+3\delta/4)R}(x_0) \;\;\text{and}\;\;|\nabla\psi|\leq \frac{C}{\delta R},
\end{equation*}
for some generic constant $ C\geq 1$ independent of $R$ and $\delta$.
For $h>0$ to be fixed later, multiplying by $e^{-\frac{|x-x_0|^2}{h}}\psi^2u$ the first equation
of Equation \eqref{6171} and integrating the latter over $\Omega\cap B_{(1+\delta)R}(x_0)$, we have
\begin{equation*}\label{s5}
\begin{split}
\frac{1}{2}\frac{d}{dt}\left(\int_{\Omega\cap B_{(1+\delta)R}(x_0)}|\psi u|^2 e^{-|x-x_0|^2/h}\,dx\right)&+\int_{\Omega\cap B_{(1+\delta)R}(x_0)}\nabla u\cdot \nabla (e^{-|x-x_0|^2/h}
\psi^2u)\,dx\\
&-\int_{\Omega\cap B_{(1+\delta)R}(x_0)}\frac{k}{|x|}
e^{-|x-x_0|^2/h}
|\psi u|^2\,dx=0.
\end{split}
\end{equation*}
From 
\begin{multline*}
\nabla (e^{-|x-x_0|^2/h}\psi^2u)=
\frac{-2(x-x_0)}{h}e^{-|x-x_0|^2/h}\psi^2u
+2e^{-|x-x_0|^2/h}\psi\nabla\psi u
+e^{-|x-x_0|^2/h}\psi^2\nabla u
\end{multline*}
and
\begin{equation*}
|\nabla(\psi u)|^2e^{-|x-x_0|^2/h}=
\psi^2|\nabla u|^2e^{-|x-x_0|^2/h}+|\nabla\psi|^2u^2e^{-|x-x_0|^2/h}
+2\psi u\nabla\psi\cdot\nabla ue^{-|x-x_0|^2/h},
\end{equation*}
it holds that
\begin{equation*}\label{ccan5}
\begin{split}
&\frac{1}{2}\frac{d}{dt}\left(\int_{\Omega\cap B_{(1+\delta)R}(x_0)}\!\!|\psi u|^2 e^{-\frac{|x-x_0|^2}{h}}
\,dx\right)
+\int_{\Omega\cap B_{(1+\delta)R}(x_0)}\!\!\Big[|\nabla(\psi u)|^2-\frac{k}{|x|}|\psi u|^2\Big]e^{-\frac{|x-x_0|^2}{h}}\,dx\\
&=\int_{\Omega\cap B_{(1+\delta)R}(x_0)}|\nabla \psi|^2u^2e^{-\frac{|x-x_0|^2}{h}}\,dx
+\frac{1}{h}\int_{\Omega\cap B_{(1+\delta)R}(x_0)}(x-x_0)\cdot
\big(\nabla(u^2)\big)\psi^2
e^{-\frac{|x-x_0|^2}{h}}\,dx.
\end{split}
\end{equation*}
By the integration by parts,
\begin{equation*}
\begin{split}
&\int_{\Omega\cap B_{(1+\delta)R}(x_0)}(x-x_0)\cdot\big(\nabla(u^2)\big)\psi^2
e^{-|x-x_0|^2/h}\,dx\\
&=-\int_{\Omega\cap B_{(1+\delta)R}(x_0)}\text{div}\big((x-x_0)\psi^2 e^{-|x-x_0|^2/h}\big)u^2\,dx\\
&=-\int_{\Omega\cap B_{(1+\delta)R}(x_0)}\!\!u^2\Big(n\psi^2 
+2\psi (x-x_0)\cdot\nabla\psi -\frac{2|x-x_0|^2}{h}\psi^2 \Big)e^{-\frac{|x-x_0|^2}{h}}\,dx\\
&=-n\int_{\Omega\cap B_{(1+\delta)R}(x_0)}|\psi u|^2 e^{-|x-x_0|^2/h}\,dx
-\int_{\Omega\cap B_{(1+\delta)R}(x_0)}2\psi u^2\nabla\psi\cdot (x-x_0) e^{-|x-x_0|^2/h}\,dx\\
&\;\;\;\;\;\;+
\frac{2}{h}\int_{\Omega\cap B_{(1+\delta)R}(x_0)}|x-x_0|^2\psi ue^{-|x-x_0|^2/h}\,dx.
\end{split}
\end{equation*}
Hence,
\begin{equation*}\label{ican5}
\begin{split}
&\frac{1}{2}\frac{d}{dt}\left(\int_{\Omega\cap B_{(1+\delta)R}(x_0)}\!\!|\psi u|^2 e^{-\frac{|x-x_0|^2}{h}}
\,dx\right)
+\int_{\Omega\cap B_{(1+\delta)R}(x_0)}\!\!\Big[|\nabla(\psi u)|^2-\frac{k}{|x|}|\psi u|^2\Big]e^{-\frac{|x-x_0|^2}{h}}\,dx\\
&\;\;=\int_{\Omega\cap B_{(1+\delta)R}(x_0)}|\nabla \psi|^2u^2e^{-|x-x_0|^2/h}\,dx-\frac{n}{h}\int_{\Omega\cap B_{(1+\delta)R}(x_0)}|\psi u|^2 e^{-|x-x_0|^2/h}\,dx\\
&\;\;\;\;\;\;-\frac{1}{h}\int_{\Omega\cap B_{(1+\delta)R}(x_0)}2\psi u^2\nabla\psi\cdot (x-x_0) e^{-|x-x_0|^2/h}\,dx\\
&\;\;\;\;\;\;+\frac{2}{h^2}\int_{\Omega\cap B_{(1+\delta)R}(x_0)}|x-x_0|^2|\psi u|e^{-|x-x_0|^2/h}\,dx\\
&\;\;\leq 2\int_{\Omega\cap B_{(1+\delta)R}(x_0)}|\nabla \psi|^2u^2e^{-|x-x_0|^2/h}\,dx-\frac{n}{h}\int_{\Omega\cap B_{(1+\delta)R}(x_0)}|\psi u|^2 e^{-|x-x_0|^2/h}\,dx\\
&\;\;\;\;\;+\frac{3}{h^2}\int_{\Omega\cap B_{(1+\delta)R}(x_0)}|x-x_0|^2|\psi u|^2e^{-|x-x_0|^2/h}\,dx.\\
\end{split}
\end{equation*}
By Lemma~\ref{morepei} (where $\varphi=\psi u$, $\lambda=h$ and $\gamma=1$), we have
\begin{equation*}
\begin{split}
&\int_{\Omega\cap B_{(1+\delta)R}(x_0)}\Big[|\nabla(\psi u)|^2-\frac{k}{|x|}|\psi u|^2\Big]e^{-|x-x_0|^2/h}\,dx\\
&\;\;\;\;\geq -\Big[\frac{n}{h}+C(k)\Big]\int_{\Omega\cap B_{(1+\delta)R}(x_0)}|\psi u|^2e^{-|x-x_0|^2/h}\,dx.
\end{split}
\end{equation*}
These last two inequalities indicate that
\begin{equation*}
\begin{split}
&\frac{1}{2}\frac{d}{dt}\left(\int_{\Omega\cap B_{(1+\delta)R}(x_0)}|\psi u|^2 e^{-|x-x_0|^2/h}
\,dx\right)
\leq \frac{3}{h^2}\int_{\Omega\cap B_{(1+\delta)R}(x_0)}|x-x_0|^2|\psi u|^2e^{-|x-x_0|^2/h}\,dx\\
&\;\;\;\;+2\int_{\Omega\cap B_{(1+\delta)R}(x_0)}|\nabla \psi|^2u^2e^{-|x-x_0|^2/h}\,dx
+C(k)\int_{\Omega\cap B_{(1+\delta)R}(x_0)}|\psi u|^2e^{-|x-x_0|^2/h}\,dx.
\end{split}
\end{equation*}
Which leads to 
\begin{equation}\label{ican6}
\begin{split}
&\frac{d}{dt}\left(\int_{\Omega\cap B_{(1+\delta)R}(x_0)}|\psi u|^2 e^{-|x-x_0|^2/h}
\,dx\right)\\
&\leq \Big[\frac{6(1+\delta)^2R^2}{h^2}+C(k)\Big]
\int_{\Omega\cap B_{(1+\delta)R}(x_0)}|\psi u|^2e^{-|x-x_0|^2/h}\,dx\\
&\;\;\;\;+\frac{C}{\delta^2R^2}
\int_{\Omega\cap (B_{(1+\delta)R}(x_0)\setminus B_{(1+\frac{3}{4}\delta)R}(x_0))}\!\!\!\!\!e^{-|x-x_0|^2/h}u^2\,dx.\\
\end{split}
\end{equation}

\bigskip

Step 2. Derive a localized energy estimate for $\int_{\Omega\cap B_R(x_0)}|u(x,T)|^2\,dx$.\\

For the simplicity of writing, we set
\begin{equation}\label{s7}
A=\frac{6(1+\delta)^2R^2}{h^2}+C(k),
\end{equation}
\begin{equation}\label{B}
B=\frac{C}{\delta^2R^2}e^{-\big(1+\frac{3}{4}\delta
\big)^2R^2/h},
\end{equation}
where $C(k)$ and $C$ are two constants in \eqref{ican6},
and
\begin{equation}\label{D}
D=\frac{\delta}{4(1+\delta)^2}.
\end{equation}
Clearly,
$$0<D< \frac{\delta}{8\delta}=\frac{1}{8}.$$
Therefore, by \eqref{ican6}, it holds that
\begin{multline*}
\frac{d}{dt}\Big(\int_{\Omega\cap B_{(1+\delta)R}(x_0)}|\psi u|^2 e^{-|x-x_0|^2/h}
\,dx\Big)\\
\leq  A\int_{\Omega\cap B_{(1+\delta)R}(x_0)}|\psi u|^2 e^{-|x-x_0|^2/h}\,dx
+Be^{C(k)}\|u_0\|_{L^2(\Omega)}^2.
\end{multline*}
Then, we have
\begin{equation*}\label{add2}
\begin{split}
\int_{\Omega\cap B_{(1+\delta)R}(x_0)}e^{-|x-x_0|^2/h}|\psi u(x,T)|^2\,dx
&\leq e^{A(T-t)}\int_{\Omega\cap B_{(1+\delta)R}(x_0)}e^{-|x-x_0|^2/h}|\psi u(x,t)|^2\,dx\\
&\;\;\;+e^{A(T-t)}Be^{C(k)}\|u_0\|_{L^2(\Omega)}^2.
\end{split}
\end{equation*}
Which  implies that when $T-Dh\leq t\leq T$,
\begin{equation*}
\begin{split}
\int_{\Omega\cap B_{(1+\delta)R}(x_0)}e^{-|x-x_0|^2/h}|\psi u(x,T)|^2\,dx
&\leq e^{ADh}\int_{\Omega\cap B_{(1+\delta)R}(x_0)}e^{-|x-x_0|^2/h}|\psi u(x,t)|^2\,dx\\
&\;\;\;+e^{ADh}Be^{C(k)}\|u_0\|_{L^2(\Omega)}^2.
\end{split}
\end{equation*}
Since $\psi(x)=1$ and $e^{-|x-x_0|^2/h}\geq e^{-R^2/h}$, $x\in \Omega\cap B_R(x_0)$, we have
\begin{equation}\label{ccan8}
\begin{split}
\int_{\Omega\cap B_{R}(x_0)}|u(x,T)|^2\,dx
&\leq e^{ADh+\frac{R^2}{h}}\int_{\Omega\cap B_{(1+\delta)R}(x_0)}e^{-|x-x_0|^2/h}| u(x,t)|^2\,dx\\
&\;\;\;+e^{ADh+\frac{R^2}{h}}Be^{C(k)}\|u_0\|_{L^2(\Omega)}^2.
\end{split}
\end{equation}
In view of \eqref{s7}, \eqref{B} and \eqref{D}, it holds that
\begin{equation}\label{ccan11}
\begin{split}
ADh+\frac{R^2}{h}&\leq\big(1+\frac{3\delta}{2}\big)\frac{R^2}{h}+C(k)
\end{split}
\end{equation}
and
\begin{equation}\label{10111}
\begin{split}
e^{ADh+\frac{R^2}{h}}Be^{C(k)}&\leq\frac{C_1(k)}{\delta^2R^2}
\exp\left\{\frac{R^2}{h}\Big[1+\frac{3\delta}{2}
-\big(1+\frac{3}{4}\delta\big)^2\Big]\right\},
\end{split}
\end{equation}
for some new constant $C_1(k)$.

\bigskip

Step 3. Fix  $h>0$.\\

By \eqref{10111}, we have
\begin{equation*}
e^{ADh+\frac{R^2}{h}}Be^{C(k)}\leq e^{-\frac{R^2}{h}\frac{9\delta^2}{16}
}\frac{C_1(k)}{\delta^2R^2}\leq e^{-\frac{\delta^2R^2}{2h}}
\frac{C_1(k)}{\delta^2R^2}.
\end{equation*}
This, along with \eqref{ccan8}, derives that the inequality
\begin{equation}\label{ccan9}
\begin{split}
\int_{\Omega\cap B_{R}(x_0)}|u(x,T)|^2\,dx
&\leq e^{ADh+\frac{R^2}{h}}
\int_{\Omega\cap B_{(1+\delta)R}(x_0)} |u(x,t)|^2\,dx\\
&\;\;+e^{-\frac{\delta^2R^2}{2h}}
\frac{C_1(k)}{\delta^2R^2}\|u_0\|_{L^2(\Omega)}^2,
\end{split}
\end{equation}
holds  when $T-Dh\leq t\leq T\leq1$.

Fix
\begin{equation*}
h=\frac{\delta^2R^2/2}{\log\Big[
\frac{C_1(k)}{\delta^2R^2}e^{\frac{1}{T}}\frac
{\|u_0\|^2_{L^2(\Omega)}}{\|u(\cdot,T)\|^2_{L^2(\Omega\cap B_R(x_0))}}\Big]},
\end{equation*}
to be such that
\begin{equation}\label{ccan10}
e^{-\frac{\delta^2R^2}{2h}}
\frac{C_1(k)}{\delta^2R^2}\|u_0\|_{L^2(\Omega)}^2=e^{-\frac{1}{T}}
\|u(\cdot,T)\|_{L^2(\Omega\cap B_R(x_0))}^2.
\end{equation}
Let $h_0=Dh$. Then $h_0$  satisfies
$h_0\in(0,T)$. Hence
it  follows from \eqref{D}, \eqref{ccan11} and  \eqref{ccan9}
that when $t\in[T-h_0,T]$,
\begin{equation}\label{valid}
\begin{split}
\big(1-e^{-\frac{1}{T}}\big)\|u(\cdot,T)\|^2_{L^2(\Omega\cap B_R(x_0))}
&\leq e^{ADh+\frac{R^2}{h}}\|u(\cdot,t)\|_{L^2(\Omega\cap B_{(1+\delta)R}(x_0))}^2\\
&\leq C(k)e^{\frac{R^2}{h}\big(1+\frac{3\delta}{2}\big)}
\|u(\cdot,t)\|_{L^2(\Omega\cap B_{(1+\delta)R}(x_0))}^2.
\end{split}
\end{equation}
 Because
\begin{equation*}
\begin{split}
&1-e^{-\frac{1}{T}}\geq \frac{1}{2}\,,\;\;\text{when}\;\;0<T\leq1,\\
&\frac{R^2}{h}\big(1+\frac{3\delta}{2}\big)
=\frac{DR^2}{h_0}\big(1+\frac{3\delta}{2}\big)< \frac{1}{h_0},
\end{split}
\end{equation*}
the desired estimate \eqref{ccan4} is valid from \eqref{valid} and \eqref{ccan10}. 
\end{proof}

\bigskip
\subsection{The proof of Theorem~\ref{global}}

We first use Lemmas~\ref{monotonicity}, \ref{hardytype}, \ref{morepei}, \ref{backwardenergyestimate} and the similar arguments as those in Lemma~\ref{peiweightestimate} to deduce the following localized version of interpolation inequality \eqref{6172}.
\begin{lemma}\label{interpolation upto boundary}
Let $0<r<R\leq 1$  and $\delta\in(0,1]$. Suppose that
$B_r(x_0)\subset\Omega$ and $\Omega\cap B_{(1+2\delta)R}(x_0)$ is star-shaped
with center $x_0\in\Omega$. Then, there exist two constants
$N=N(\delta,R,k)\geq 1$ and $\theta=\theta(\delta,R,r,k)$ with $\theta\in (0,1)$,
such that  for any solution $u$  to Equation~(\ref{6171}) and any  $0<T\leq 1$,
\begin{equation}\label{interpolationi}
\int_{\Omega\cap B_R(x_0)}|u(x,T)|^2\;dx\leq
\left(Ne^{\frac{N}{T}}\int_{\Omega}|u(x,0)|^2\;dx\right)^{\theta}
\left(\int_{B_r(x_0)}|u(x,T)|^2\;dx\right)^{1-\theta}.
\end{equation}
\end{lemma}

\begin{proof}
We only need to prove the desired estimate \eqref{interpolationi} for an arbitrarily fixed solution  $u$  to
Equation \eqref{6171} with $u(\cdot, 0)\neq 0$ in $L^2(\Omega)$.
Let $\psi\in C_0^2(\Omega\cap B_{R+2\delta R}(x_0))$ be the cut-off function such that $0\leq\psi\leq1$,
$$\psi=1\;\;\text{in}\;\;\Omega\cap B_{R+\frac{3\delta}{2} R}(x_0),\;\;
|\nabla\psi|^2+|\Delta \psi|\leq\frac{C}{\delta^2R^2},$$
where the generic constant $C$ is independent of $R$ and $\delta$.
Let $z=\psi u$ and
\begin{equation}\label{2918}
f=\partial_tz-\Delta z=\frac{k}{|x|}z-2\nabla\psi\cdot\nabla u-\Delta\psi u.
\end{equation}
Associated with the triple $(z,f,x_0)$, we set for each $\lambda>0$ and  $t\in(0,T]$,
\begin{equation*}
H_\lambda(t)=\int_{\Omega\cap B_{(1+2\delta)R}(x_0)}z^2G_{\l,x_0}\,dx
\end{equation*}
and
\begin{equation*}
N_\lambda(t)=\frac{\int_{\Omega\cap B_{(1+2\delta)R}(x_0)}|\nabla z|^2G_{\l,x_0}\,dx}{\int_{\Omega\cap B_{(1+2\delta)R}(x_0)}z^2
G_{\l,x_0}\,dx}.
\end{equation*}

\bigskip

Step 1.
For each $t\in[T-D_1h_0,T]$, where $h_0$ is given by \eqref{ch} and  $0<D_1\leq 1$ is to be fixed later, we obtain
from Lemmas~\ref{hardytype}, \ref{morepei} and the similar argument as in Lemma~\ref{peiweightestimate}, as well as Lemma~\ref{backwardenergyestimate} (cf. \cite[Lemma 4, Step 1]{PWZ} ) that
\begin{equation}\label{21-4}
\frac{\int_{\Omega\cap B_{R+2\delta R}(x_0)}zfG_{\l,x_0}\,dx}
{\int_{\Omega\cap B_{R+2\delta R}(x_0)}z^2G_{\l,x_0}\,dx}
\leq \frac{1}{2}N_\lambda(t)
+\Big[\frac{n}{8(T-t+\lambda)}+C(2k)\Big]+
\frac{C}{\delta^2R^2}e^{-\frac{\delta^2 R^2}{D_1h_0+\l}}
\big(1+T^{-1/2}\big)e^{\frac{2}{h_0}}
\end{equation}
and
\begin{equation}\label{21-3}
\frac{\int_{\Omega\cap B_{R+2\delta R}(x_0)}f^2G_{\l,x_0}\,dx}
{\int_{\Omega\cap B_{R+2\delta R}(x_0)}z^2G_{\l,x_0}\,dx}
\leq \frac{3k^2}{\mu^*}
N_\lambda(t)+\frac{3nk^2}{4\mu^*(T-t+\lambda)}+
\frac{C}{\delta^4 R^2}e^{-\frac{\delta^2 R^2}{D_1h_0+\l}}
\big(1+T^{-1/2}\big)e^{\frac{2}{h_0}}.
\end{equation}

\bigskip

Step 2. It follows from  \cite[Lemma 2.5 ]{pw2}
that
\begin{align*}
&\int_{\Omega\cap B_{R+2\delta R}(x_0)}|z(x,T)|^2e^{-\frac{|x-x_0|^2}{4\l}}\,dx\\
&\leq \int_{B_{r}(x_0)}|u(x,T)|^2
e^{-\frac{|x-x_0|^2}{4\l}}\,dx
+\frac{16\l}{r^2}\Big[\l N_\l(T)+\frac{n}{4}\Big]
\int_{\Omega\cap B_{R+2\delta R}(x_0)}|z(x,T)|^2
e^{-\frac{|x-x_0|^2}{4\l}}\,dx.
\end{align*}

Step 3. By  \eqref{21-3} and (ii)  of Lemma~\ref{monotonicity}
(where $u=z$ and $f$ is given by \eqref{2918}), we have that for each $t\in[T-D_1h_0,T]$,
\begin{equation*}\label{21-2-2}
\frac{d}{dt}N_\l(t)\leq \frac{N_\l(t)}{T-t+\lambda}+\frac{3k^2}{\mu^*}
N_\lambda(t)+\frac{3nk^2}{4\mu^*(T-t+\lambda)}+
\frac{C}{\delta^4 R^2}e^{-\frac{\delta^2 R^2}{D_1h_0+\l}}
\big(1+T^{-1/2}\big)e^{\frac{2}{h_0}}.
\end{equation*}
Consequently,
\begin{equation*}
\frac{d}{dt}\big[(T-t+\lambda)N_\l(t)\big]\leq
\frac{3k^2}{\mu^*}(T-t+\l)N_\l(t)+\frac{3nk^2}{4\mu^*}
+\frac{C}{\delta^4 R^2}e^{-\frac{\delta^2 R^2}{D_1h_0+\l}}
\big(1+T^{-1/2}\big)e^{\frac{2}{h_0}}(T+\lambda).
\end{equation*}
It follows from the above inequality  that when $t\in[T-D_1h_0,T]$,
\begin{equation*}
\begin{split}
e^{-3Tk^2/\mu^*}\l N_\l(T)&\leq
(T-t+\l)N_\l(t)+\frac{3nk^2}{4\mu^*}
+\frac{C}{\delta^4 R^2}e^{-\frac{\delta^2 R^2}{D_1h_0+\l}}
\big(1+T^{-1/2}\big)e^{\frac{2}{h_0}}(T+\lambda).
\end{split}
\end{equation*}
Integrating  the above inequality with respect to $t$ over $[T-D_1h_0,T-D_1h_0/2]$, we obtain that
\begin{multline}\label{pei11-c}
e^{-3Tk^2/\mu^*}\l N_\l(T)\leq \frac{2(D_1h_0+\l)}{D_1h_0}
\int_{T-D_1h_0}^{T-D_1h_0/2}N_\l(t)\,dt\\+
\frac{3nk^2}{4\mu^*}+\frac{C}{\delta^4 R^2}e^{-\frac{\delta^2 R^2}{D_1h_0+\l}}
\big(1+T^{-1/2}\big)e^{\frac{2}{h_0}}(T+\lambda).
\end{multline}
On the other hand, we see from \eqref{21-4} and  (i) in Lemma~\ref{monotonicity} (where $u=z$ and $f$ is given by \eqref{2918}) that
\begin{equation*}
N_\l(t)\leq -\frac{d}{dt}\log \big(H_\l(t)\big)+
\Big[C(2k)+\frac{n}{4(T-t+\l)}
\Big]+\frac{C}{\delta^2}e^{-\frac{\delta^2 R^2}{D_1h_0+\l}}
\big(1+T^{-1/2}\big)e^{\frac{2}{h_0}}.
\end{equation*}
Integrating the above inequality over $[T-D_1h_0,T-D_1h_0/2]$, we get that
\begin{equation}\label{pei9-c}
\int_{T-D_1h_0}^{T-D_1h_0/2}N_\l(t)\,dt\leq\log \frac{H_\l(T-D_1h_0)}{H_\l(T-D_1h_0/2)}+\frac{T}{2}\Big[C(k)+\frac{n}{2T}
\Big]+
\frac{C}{\delta^2}e^{-\frac{\delta^2 R^2}{D_1h_0+\l}}
\big(1+T^{-1/2}\big)e^{\frac{2}{h_0}}.
\end{equation}
Notice from Lemma~\ref{backwardenergyestimate} that
\begin{equation}\label{pei10-c}
\frac{H_\l(T-D_1h_0)}{H_\l(T-D_1h_0/2)}\leq e^{C(k)+\frac{(R+\delta R)^2}{2D_1h_0}+\frac{2}{h_0}}.
\end{equation}
Therefore, it follows from \eqref{pei11-c}, \eqref{pei9-c} and  \eqref{pei10-c} that when $0<\l\leq D_1h_0$,
\begin{equation*}\label{pei12-c}
\begin{split}
\l N_\l(T)&\leq
4e^{\frac{3k^2}{\mu^*}}
\left[\frac{(R+\delta R)^2}{2D_1h_0}+\frac{2}{h_0}+C(k)+\frac{n}{2}\big(
1+\frac{3k^2}{\mu^*}\big)+\mathcal{M}\right]
\end{split}
\end{equation*}
with
$$\mathcal{M}=\frac{C}{\delta^4 R^2}e^{-\frac{\delta^2 R^2}{D_1h_0+\l}}
\big(1+T^{-1/2}\big)e^{\frac{2}{h_0}}
+\frac{C}{\delta^2}e^{-\frac{\delta^2 R^2}{D_1h_0+\l}}
\big(1+T^{-1/2}\big)e^{\frac{2}{h_0}}.$$

\bigskip

Step 4. Let
$$D_1=\delta^2 R^2\;\;\;\text{and}\;\;\l=\varepsilon D_1h_0.$$
Here $\varepsilon\in(0,1)$ will be choose later.
Then
\begin{equation*}
\frac{2}{h_0}-\frac{\delta^2 R^2}{D_1h_0+\l}
\leq 0,
\end{equation*}
and consequently
\begin{equation*}
\mathcal M\leq \frac{C}{\delta^4 R^2}\big(1+T^{-1/2}\big).
\end{equation*}
Therefore, by Step 4 and noticing that $0<h_0<T\leq1$, we have
\begin{eqnarray*}
\frac{16\l}{r^2}\Big[\l N_\l(T)+\frac{n}{4}\Big]
&\leq& \frac{\varepsilon D_1}{r^2}e^{\frac{3k^2}{\mu^*}}
\Big[C(\mu^*,k,n)h_0+D_1h_0 \frac{C}{\delta^4R^2}\big(1+T^{-1/2}\big)
+\frac{(R+\delta R)^2}{2D_1}
\Big]\\
&\leq&\varepsilon C(\delta,R,r,k,\mu^*).
\end{eqnarray*}

\bigskip

Step 5. Finally, we choose $\varepsilon\in(0,1)$, which is independent of $T$, such that
\begin{align*}
\frac{16\l}{r^2}\Big[\l N_\l(T)+\frac{n}{4}\Big]\leq
 \varepsilon C(\delta,R,r,k,\mu^*)=\frac{1}{2}.
\end{align*}
Hence, by Step 2, we get that
\begin{align*}
\int_{\Omega\cap B_{R+2\delta R}(x_0)}|z(x,T)|^2e^{-\frac{|x-x_0|^2}
{4\l}}\,dx
\leq 2\int_{B_{r}(x_0)}|u(x,T)|^2e^{-\frac{|x-x_0|^2}{4\l}}\,dx,
\end{align*}
which implies that
\begin{align}\label{dudu}
\int_{\Omega\cap B_{R}(x_0)}|u(x,T)|^2\,dx
\leq e^{\frac{R^2}{4\l}}
\int_{B_{r}(x_0)}|u(x,T)|^2\,dx.
\end{align}
Recall that
$$\l=\varepsilon D_1 h_0=
\frac{C(r,R,\delta,k,\mu^*)}{
\log\Big[\frac{C}{\delta^2 R^2}e^{\frac{1}{T}}\frac{\|u(\cdot,0)\|^2_{L^2(\Omega)}}{\|u(\cdot,T)\|
^2_{L^2(\Omega\cap B_R(x_0))}}\Big]}\;.
$$
Thereby,
$$e^{\frac{R^2}{4\l}}=
\Big[\frac{C}{\delta^2 R^2}e^{\frac{1}{T}}\frac{\|u(\cdot,0)\|^2_{L^2(\Omega)}}{\|u(\cdot,T)\|
^2_{L^2(\Omega\cap B_R(x_0))}}\Big]^{C(r,R,\delta,k,\mu^*)}.
$$
Which, together with \eqref{dudu},  completes the proof.
\end{proof}

\bigskip

 At the end, by successively making use  of Lemma~\ref{interpolation upto boundary} and an argument of propagation of smallness (see, e.g., \cite{PWZ}), we have the following.

\begin{proof}[The proof of Theorem~\ref{global}]
Firstly, by Lemma~\ref{interpolation upto boundary}  and
constructing a sequence of balls chained along a curve, we claim that, for any compact sets $K_1$ and $K_2$
with non-empty interior in $\Omega$, there are constants
$N_1=N_1(K_1,K_2,k,n)\geq 1$ and $\alpha_1=\alpha_1(K_1,K_2,k,n)\in(0,1)$
such that
\begin{equation}\label{interior}
\int_{K_1}|u(x,T)|^2\;dx\leq \left(N_1e^{\frac{N_1}{T}}
\int_{\Omega}|u(x,0)|^2\;dx\right)^{1-\alpha_1}
\left(\int_{K_2}|u(x,T)|^2\;dx\right)^{\alpha_1}.
\end{equation}
Indeed, let
\begin{equation}\label{covering}
K_1\subset\bigcup_{i=1}^{p}B_{r}(x_i)\subset\Omega, \;\;\;\;
B_{r}(x_0)\subset K_2,
\end{equation}
and for each $B_{r}(x_i)$ with $1\leq i\leq p$, there exists a chain
of balls $B_{r}(x_i^j)$, $1\leq j\leq n_i$, such that
\begin{equation}\label{chain}
\begin{split}
&B_r(x_i^{1})=B_{r}(x_i),\;\;B_r(x_i^{n_i})=B_{r}(x_0),\\
&B_{r}(x_i^j)\subset B_{2r}(x_i^{j+1})\subset\Omega,\;\;1\leq j\leq
n_i-1,
\end{split}
\end{equation}
where $0<r=r(\om,K_1,K_2)<1/2$ is a constant. By
Lemma~\ref{interpolation upto boundary},  we get that there are
$N_i^j= N_i^j(r,k,n)\geq 1$ and $\theta_i^j=\theta_i^j(r,k,n)\in(0,1)$ such
that
\begin{equation*}
\int_{B_{2r}(x_i^{j+1})}|u(x,T)|^2\;dx\leq
\left(N_i^je^{\frac{N_i^j}{T}}\int_{\Omega}|u(x,0)|^2\;dx\right)
^{\theta_i^j}
\left(\int_{B_r(x_i^{j+1})}|u(x,T)|^2\;dx\right)^{1-\theta_i^j}.
\end{equation*}
This, together with \eqref{chain}, implies that there are
$N_i=N_i(K_1,K_2,k,n)\geq 1$ and $\theta_i=\theta_i(K_1,K_2,k,n)\in(0,1)$
such that
\begin{equation*}
\int_{B_{r}(x_i)}|u(x,T)|^2\;dx\leq
\left(N_ie^{\frac{N_i}{T}}\int_{\Omega}|u(x,0)|^2\;dx\right)
^{\theta_i} \left(\int_{B_r(x_0)}|u(x,T)|^2\;dx\right)^{1-\theta_i}.
\end{equation*}
This, along with (\ref{covering}) leads to the estimate \eqref{interior}.

Secondly, since $\partial\Omega$ is of class $C^2$,  there are a
finite set $ \Lambda\subset \Omega$, $0<\delta<1$ and a family of
positive numbers $0<r_x\le 1$, $x\in\Lambda$, such that
\begin{equation*}
\partial\Omega\subset \bigcup_{x\in \Lambda} B_{r_x}(x)\;\;\;\mbox{and}\;\;\; B_{\left(1+2\delta\right)r_x}(x)\cap\Omega\ \text{is}\ \text{star-shaped with center}\;x.
\end{equation*}
Then we apply Lemma~\ref{interpolation upto boundary} with
$\Omega\cap B_{(1+2\delta)r_x}(x),x\in\Lambda$, and the same
arguments as above to get that, when $\Gamma$ is a neighborhood of
$\partial\Omega$ and $K_3$ is a compact set with non-empty interior
in $\Omega$, there are constants $N_2=N_2(\Gamma,K_3,k,n)\geq 1$ and
$\alpha_2=\alpha_2(\Gamma,K_3,k,n)\in(0,1)$ such that
\begin{equation*}
\int_{\Gamma}|u(x,T)|^2\;dx\leq \left(N_2e^{\frac{N_2}{T}}
\int_{\Omega}|u(x,0)|^2\;dx\right)^{1-\alpha_2}
\left(\int_{K_3}|u(x,T)|^2\;dx\right)^{\alpha_2}.
\end{equation*}

Finally, we derive the desire estimate \eqref{cccc1} from the
previous two statements with $\Omega\subset(\Gamma\cup K_1)$ and
$(K_2\cup K_3)\subset\omega$.
\end{proof}

\medskip
\section{Conclusion and further comments}\label{sec4}

In the present paper, by adapting  the frequency function method in \cite{Phung-Wang1, PWZ}, we have derived a H\"older-type quantitative estimate of unique continuation  for any solution to the heat equation with singular Coulomb potentials in either a bounded Lipschitz domain or a bounded domain  with a $C^2$-smooth boundary. 

Several remarks are given in order.

\begin{itemize}

\item[1.] \textit{Applications in Control Theory}.
As addressed in the introduction, such a  kind of quantitative estimate of  unique continuation has been proved to be applicable in the subject of control theory in recent years (see e.g. \cite{AEWZ,EMS,EMS2,pw2,pwx,wz,WZ,WY,Yu,ZB}).
In particular, it can be used to establish the null controllability from measurable subsets of positive measure, and to obtain the bang-bang property of time and norm optimal control problems for the heat equation with singular Coulomb 
potentials.

\medskip

\item[2.] \textit{Variable Coefficients}.
In the main theorems of this paper, we established the quantitative estimate of  unique continuation for the heat equation
with a singular  potential $k/|x|$, where $k$ is a constant. If we allow  that $k$ is a bounded function depending on both space and time variables instead of 
being a constant, then the corresponding quantitative estimate for the heat equation
with a singular  potential $|x|^{-1}k(x,t)$ could  also be  obtained by using the same arguments as in the current  paper with some minor modifications.  The details are left to the interested reader.
Notice that quantitative estimates of strong unique continuation for second order parabolic equations with bounded potentials have been obtained in \cite{BP, EFV,Phung-Wang1,pw2,PWZ}), by either the Carleman estimate  method or the frequency function method.
However, it is still an open question whether it could  be improved for the general second-order parabolic equation
with singular Coulomb potentials.

\bigskip

\item[3.]  \textit{Inverse Square Potentials}. 
Another question is whether one can still expect results as in Theorems \ref{g-t} and \ref{global} if the Coulomb 
potential in \eqref{6171} is replaced by the inverse square potential $\mu/|x|^2$, which is known as the most singular lower-order potential from the viewpoint of well-posedness and unique continuation (cf., e.g., \cite{VZ} and \cite{Lin1}).  The strong unique continuation at the singularity point  for the heat equation with
inverse square potentials has been obtained in \cite{Poon} and \cite{Felli} (see also \cite{O}).
By  Carleman estimates associated with the heat operator with
inverse square potentials,
the observability inequality from an observation region,  which is  away from the singularity,    for the heat equation with
inverse square potentials has also been obtained in \cite{er}
and \cite{vz}.   
Because of the strong singularity near the origin, however, our present methods
do not allow us to derive the quantitative strong unique continuation estimate \eqref{6172} for the heat equation
with this kind of  inverse square potential, although we dare to conjecture that it should be possible. 

\end{itemize}

\bigskip

\textbf{Acknowledgments}.
The author would like to thank Prof. G. Wang for drawing the author's attention on this subject, and Prof. L. Escauriaza and Prof. K. D. Phung for fruitful discussions we had related to this work.
The author acknowledge the financial support by the National Natural Science Foundation of China under grants 11501424.


\begin{thebibliography}{12}


\bibitem{AEWZ} J. Apraiz, L. Escauriaza, G. Wang, C. Zhang. Observability inequalities and measurable sets.
   J. Eur. Math. Soc., 16 (2014) 2433-2475.  
   
\bibitem{BP} C. Bardos, K. D. Phung,
 Observation estimate for kinetic transport equation by diffusion approximation,
Comptes Rendus Mathematique,
Volume 355 (2017),  640--664.

\bibitem{CXY}
X. Y. Chen,  A strong unique continuation theorem for parabolic equations. Mathematische Annalen, 311 (1998), 603-630.


\bibitem{evans} L. C. Evans,
Partial Differential Equations,
Volume 19 of Graduate Studies in Mathematics, 
American Mathematical Soc., 2010.

\bibitem{er}
S. Ervedoza, Control and stabilization properties for a singular heat equation with an inverse-square potential, Communications in Partial Differential Equations, 33 (2008), 1996--2019.

   
\bibitem{EMS} L. Escauriaza, S. Montaner, C. Zhang. Observation from measurable sets for parabolic analytic evolutions and applications. 
J. Math. Pures Appl. 104 (2015) 837-867.

\bibitem{EMS2} L. Escauriaza, S. Montaner, C. Zhang, Analyticity of solutions to parabolic evolutions and applications.  To appear in SIAM Journal on Mathematical Analysis.  

 
\bibitem{E1} L. Escauriaza. Carleman inequalities and the heat operator. Duke Math. J.,  104 (2000), 113-127.

\bibitem{EF} L. Escauriaza, F.J. Fern\'andez. Unique continuation for parabolic operators. Ark. Mat.,  41 (2003), 35-60.

\bibitem{EFV} L. Escauriaza, F. J. Fern\'andez, S. Vessella, Doubling properties of caloric functions.  Appl. Anal., 85 (2006),  205--223.

\bibitem{F} F. J. Fern\'andez. Unique continuation for parabolic operators II. Communications in Partial Differential Equations,  28 (2003), 1597--1604.

\bibitem{Felli} V. Felli,  A. Primo, Classification of local asymptotics for solutions to heat equations with inverse-square potentials. Discrete and Continuous Dynamical Systems (DCDS-A), 31 (2011),  65-107.


\bibitem{Lin1} N. Garofalo, F. H. Lin, Monotonicity properties
of variational integrals: $A_p$ weights and unique continuation.
Indiana University Math. J., 35 (1986), 245-268.

\bibitem{Lin2} N. Garofalo, F. H. Lin, Unique continuation
for elliptic operators: a geometric-variation approach.
Comm. Pure. Appl. Math, 40 (1987), 347-366.



\bibitem{K} I. Kukavica, K. Nystr\"om. Unique continuation on the boundary for Dini domains. Proc. Amer. Math. Soc., 126 (1998), 441--446.


\bibitem{L} F. H. Lin,  A uniqueness theorem for parabolic equations.
Comm. Pure. Appl. Math, 43 (1990), 127-136.

\bibitem{qi1}
Q. L\"u,  Z. Yin, Unique continuation for stochastic heat equations. ESAIM Control
Optim. Calc. Var. 21 (2015), 378-398.

\bibitem{qi2}
Q. L\"u, Strong unique continuation property for
stochastic parabolic equations.  https://arxiv.org/pdf/1701.02136.pdf. 2017.



\bibitem{Lions2} J. L. Lions, Optimal Control of Systems Governed by Partial Differential Equations. Springer-Verlag, Berlin Heildeberg New York, 1971.

\bibitem{O} T. Okaji, A note on unique continuation for parabolic operators with singular potentials,  Studies in Phase Space Analysis with Applications to PDEs. Springer New York, 2013: 291-312.




\bibitem{Poon} C. C. Poon, Unique continuation for parabolic equations.  Communications in Partial Differential Equations, 21 (1996), 521-539.


\bibitem{Phung-Wang1} K. D. Phung, G. Wang, Quantitative unique
continuation for the semilinear heat equation in a convex domain,
J. Funct. Anal., 259 (2010), 1230-1247.

\bibitem{pw2} K. D. Phung, G. Wang. An observability estimate for parabolic equations from a measurable set in time and its applications,  J. Eur. Math. Soc., 15
    (2013), 681-703.
    
    
\bibitem{pwx} K. D. Phung, G. Wang, Y. Xu, Impulse output rapid stabilization for heat equations, arXiv:1611.10075. To appear in Journal of Differential Equations, 2017. 

\bibitem{PWZ} K. D. Phung, L. Wang, C. Zhang, Bang-bang property for time optimal control of semilinear heat equation.
 Annales de I'Institut Henri Poincare (C) Non Linear Analysis, 31 (2014),  477-499


\bibitem{VZ} J. L. Vazquez, E. Zuazua, The Hardy inequality and the asymptotic behaviour of the heat equation with an inverse-square potential. J. Funct. Anal., 173 (2000), 103-153.

 \bibitem{vz}
J. Vancostenoble, E. Zuazua, Null controllability for the heat equation with singular inverse-square potentials. J. Funct. Anal., 254 (2008), 1864-1902.

\bibitem{V1} S. Vessella, Unique continuation properties and quantitative estimates of unique continuation for parabolic equations. Handbook of Differential Equations: Evolutionary Equations, Volume 5, 2009, 423-500.

\bibitem{wz} G. Wang, C. Zhang,
Observability inequalities from measurable sets for some abstract evolution equations.
SIAM Journal on Control and Optimization, 55 (2017), 1862-1886.

\bibitem{WZ} G. Wang, E. Zuazua, On the equivalence of minimal time and minimal norm controls for internally controlled heat equations,  SIAM Journal on Control and Optimization 50 (2012), 2938-2958.



\bibitem{WY}
L. Wang, Y. Qishu, Bang-bang property of time optimal null controls for some semilinear heat equation. SIAM Journal on Control and Optimization 54 (2016), 2949-2964.



\bibitem{Yu} H. Yu, Approximation of time optimal controls for heat equations with perturbations in the system potential, SIAM J. Control Optim., 52 (2014), 1663-1692.

\bibitem{xu}
X. Zhang, Unique continuation for stochastic parabolic equations. Differential Integral
Equations. 21 (2008), 81-93.


\bibitem{ZB} Y. Zhang, Two equivalence theorems of different kinds of optimal control problems for
Schr\"{o}dinger equations,  SIAM J. Control Optim.,  53 (2015),  926-947.


\end{thebibliography}
\end{document}